\documentclass[12pt,a4paper,draft,reqno]{amsproc}

\usepackage[cp1251]{inputenc}
\usepackage[english, russian]{babel}

\usepackage{amssymb}
\usepackage{a4wide}
\usepackage{mathrsfs,euscript}

\theoremstyle{plain}
\newtheorem{theorem}{\bf Theorem} 
\newtheorem{lemma}{\bf Lemma} 

\newtheorem{problem}{\bf Problem}

\theoremstyle{definition}

\theoremstyle{plain}
\newtoks\thehProclaim
\newtheorem*{Proclaim}{\the\thehProclaim}

\theoremstyle{definition}
\newtoks{\thehRemark}
\newtheorem*{Remark}{\the\thehRemark}

\DeclareRobustCommand{\No}{%
\ifmmode{\nfss@text{\textnumero}}\else\textnumero\fi}

\begin{document}


\def\contentsname{Table of Contents}
\def\figurename{Figure}
\def\partname{Part}
\def\refname{References}
\def\bibname{References}
\def\tablename{Table}
\def\chaptername{Chapter}
\def\proofname{Proof}
\def\appendixname{Appendix}
\def\keywordsname{Keywords}


\title{The perturbation method for the skew-symmetric strongly elliptic systems of PDEs}

\author[A.\,O.~Bagapsh]{А.\,О.~Bagapsh\textsuperscript{1),} \textsuperscript{2)}}

\maketitle


For a Jordan domain with sufficiently smooth boundaries, the solution to the Dirichlet problem for second order
skew-symmetric strongly elliptic system with constant coefficients and regular enough boundary data is constructed in the form of a
power series of a small parameter describing the perturbation of the given system from the Laplace one. The coefficients of this
series are the functions that are determined sequentially as solutions to special Dirichlet problems for the usual Laplace and Poisson equations. The obtained series converges uniformly in the closure of the domain under consideration.

\textbf{Keywords:} elliptic system; Dirichlet problem; perturbation method.

\markright{The perturbation method}

\footnotetext[1]
{Federal Research Center "`Computer Science and Control"' of the Russian Academy of Sciences, Moscow, Russia;
St.Petersburg State University, St.Petersburg, Russia;
Moscow Center for Fundamental and Applied Mathematics, Moscow, Russia \\
e-mail: a.bagapsh@gmail.com}

\footnotetext[2]{This research was supported by the Grants Council of the President of the Russian Federation, grant MK-1204.2020.1, by Ministry of Science and Higher Education of the Russian Federation, project 0705-2020-0047, and by the Theoretical Physics and Mathematics Advancement Foundation «BASIS».}

\section{Introduction and the main result}

In the present paper we are dealing with the Dirichlet problem for second order\break homogeneous elliptic equations with constant
\emph{complex} coefficients in Jordan domains in the complex plane $\mathbb C$. Under certain regularity assumptions on the boundaries of domains under consideration and on the boundary functions, we present a solution representation in the form of a special power series of a small parameter
describing the perturbation of a given equation from the Laplace one, and prove that the corresponding series converges uniformly in the closure
of the domain under consideration.

Consider the following partial differential equation
\begin{equation}\label{el-eq}
af_{xx}+2bf_{xy}+cf_{yy}=0
\end{equation}
on the complex--valued function $f$ of the complex variable $z=x+iy$ with constant coefficients $a$, $b$, $c\in\mathbb C$. Denote $a=a_1+ia_2$, $b=b_1+ib_2$, $c=c_1+ic_2$. The equation \eqref{el-eq} can be written as a system of equations
\begin{equation}\label{el-sys}
\left(A\frac{\partial^2}{\partial x^2}+2B\frac{\partial^2}{\partial x\partial y}+C\frac{\partial^2}{\partial y^2}\right)
\left(\begin{matrix} u \\ v \end{matrix}\right)=\left(\begin{matrix} 0 \\ 0 \end{matrix}\right)
\end{equation}
on the real and imaginary parts $u=\text{Re} f$, $v=\text{Im} f$, where
\begin{equation}\label{matr}
A=\left(\begin{matrix}a_1 & -a_2 \\a_2 & a_1  \end{matrix}\right),\qquad
B=\left(\begin{matrix}b_1 & -b_2 \\b_2 & b_1  \end{matrix}\right),\qquad
C=\left(\begin{matrix}c_1 & -c_2 \\c_2 & c_1 \end{matrix}\right)
\end{equation}
are the constant real skew-symmetric matrices. Through all this paper the equation \eqref{el-eq} is assumed to be elliptic, which means that
$$
\det(A\xi^2+2B\xi\eta+C\eta^2)\ne 0\quad\text{when}\quad\xi^2+\eta^2>0.
$$
Moreover, we will also assume that the equation \eqref{el-eq} is strongly elliptic. The latter property means that
$$
\det(A+2\alpha B+\beta C)\ne 0\quad\text{for each}\quad\alpha\leqslant\beta^2,
$$
see, for instance \cite{Petrovskiy-39}--\cite{HuaLinWu-85}. The conditions of ellipticity and strong ellipticity given above are stated
in terms of the system \eqref{el-sys} that is equivalent to the equation \eqref{el-eq}. It is worth to show how these conditions look line in terms of the initial equation \eqref{el-eq}. In such terms the ellipticity property means that both roots of the corresponding characteristic equation $a\lambda^2+2b\lambda+c=0$ are not real, while the strong ellipticity property means that these roots belong to the different half-planes of $\mathbb C$ with respect to the real axis.

Using some suitable affine transformations of the independent variables ($x$ and $y$) and of the dependent variables ($u$ and $v$), and
some suitable linear combination of equations in system \eqref{el-sys}, one can transform the initial system to the system corresponding to the equation $\mathcal Lf=0$ on new complex valued function $f$ of the new complex variable $z$ with the operator
\begin{equation}\label{L}
\mathcal L=\partial\overline\partial+\tau\partial^2,\qquad\tau\in[0,1),
\end{equation}
see \cite{HuaLinWu-85} and \cite{BagFed-17}, where
\begin{equation*}
\overline\partial=\frac{1}{2}\left(\frac{\partial}{\partial x}+i\frac{\partial}{\partial y}\right),\qquad
\partial=\frac{1}{2}\left(\frac{\partial}{\partial x}-i\frac{\partial}{\partial y}\right)
\end{equation*}
are the standard Cauchy--Riemann differential operator and its complex conjugate. We will also use the notation $\overline\partial_z$
and a$\partial_z$ for them. Note that when $\tau=0$ the equation $\mathcal Lf=0$ turns out to be the Laplace equation $\Delta f=0$, because
$\Delta=4\partial\overline\partial$.

We are interested in the question on solutions construction to the Dirichlet problem for the equation $\mathcal Lf=0$ in its classic
setting. Let us recall the statement of this problem.
\begin{problem}\label{dirprob}
Let $\varOmega$ be a bounded simply connected domain, and let $\varGamma=\partial\varOmega$ be its boundary. Given a function $h\in C(\varGamma)$, find a function $f\in C(\overline\varOmega)$ such that $\mathcal Lf=0$ in $\varOmega$ and $f=h$ on $\varGamma$.
\end{problem}

The most important and intriguing question related with Problem~\ref{dirprob} is to describe such domain $\varOmega$ where the
corresponding Dirichlet problem for the equation $\mathcal Lf=0$ is solvable for every continuous boundary function $h$. This question remains unsolved in the general setting, even it is not known, whether Problem~\ref{dirprob} is solvable for an arbitrary boundary boundary function $h$ data in a Jordan domain of general type. The most general known result (at the best of our knowledge) is that any Dirichlet problem for all
$\mathcal L$ under consideration is always solvable in $C^1$--polygons, see \cite{VerVog-97}.

In this paper we are dealing with the special explicit method of solutions constructing to Problem~\ref{dirprob}, which is based on the solution
representation of the form
\begin{equation}\label{series}
f=\sum\limits_{n=0}^\infty f_n\tau^n.
\end{equation}
Using this method we answer the question on solvability of Problem~\ref{dirprob} with H\"older boundary data in Jordan domains with sufficiently
smooth boundaries, and present new explicit formulae for such solutions.

Our main result, Theorem~\ref{dirprob-solv}, states that if $\varOmega$ is a Jordan domain and if the function $h\in C(\varGamma)$ is
such that $h\circ\omega$ belongs to the H\"older class $C^\alpha(\mathbb T)$ for some $1/2<\alpha<1$ and some conformal mapping $\omega$ of the unit disk $\mathbb D$ to $\varOmega$, then for each $\tau\in[0,1)$ there exists a solution of the Dirichlet problem that is represented in the form \eqref{series}, converging uniformly on the closure $\overline\varOmega$ of $\varOmega$.

Observe that the condition $h\circ\omega\in C^\alpha(\mathbb T)$ is satisfied, for instance, when $h\in C^\alpha(\varGamma)$ and
$\varGamma$ is a Diny smooth curve, where $\varGamma$ is the boundary of $\varOmega$, as above. Indeed, by Lindel\"of's theorem
$\omega'\in C(\overline{\mathbb D})$, so for each $z_1$, $z_2\in\mathbb T$ it holds
$$
|h\circ\omega(z_1)-h\circ\omega(z_2)|\leqslant[h]_\alpha\,|\omega(z_1)-\omega(z_2)|^\alpha\leqslant[h]_\alpha\,
\|\omega'\|_{C(\overline{\mathbb D})}^\alpha
\,|z_1-z_2|^\alpha,
$$
where $[h]_\alpha=\sup_{\zeta_1\ne\zeta_2}|h(\zeta_1)-h(\zeta_2)|/|\zeta_1-\zeta_2|^\alpha$.

The paper has the following structure. In Section~2 we explain the aforementioned method of solutions constructing in the form of representations
\eqref{series}. Further, in Section~3 we prove the convergence of the series \eqref{series} under assumptions to $h$ and $\varOmega$ mentioned
above.

\section{Perturbation method}

Let $\Omega$ be a Jordan domain with the boundary $\varGamma$. Take a function $h\in C(\varGamma)$. We are going to obtain the
function $f$ of the form \eqref{series} such that $\mathcal Lf=0$ and $f|_{\varGamma}=h$. To do this we substitute the expansion \eqref{series}
into the equation $\mathcal Lf=0$ and equate to zero the factors $f_n$ for the same powers of the parameter $\tau$. Moreover, we put
$f_0|_\varGamma=h$ and $f_n|_\varGamma=0$ for $n\ge 1$. Therefore, the conditions on $f_n$, $n\ge0$, are reduced to the following sequence of boundary value problems:
\begin{equation}\label{f0-prob}
\partial\overline\partial f_0=0\quad\text{in}\quad\varOmega,\qquad\qquad f_0|_\varGamma=h
\end{equation}
and
\begin{equation}\label{fn-prob}
\partial\overline\partial f_n=-\partial^2 f_{n-1}\quad\text{in}\quad\varOmega,\qquad\qquad f_n|_\varGamma=0
\end{equation}
for $n\geqslant 1$.

Let $\omega\colon\mathbb D\to\varOmega$ be some conformal mapping from the unit disk $\mathbb D=\{z\in\mathbb C\colon|z|<1\}$ onto $\varOmega$. If
$\varOmega$ is Jordan then, according to the Carath\'eodory extension theorem \cite[Theorem~2.6]{Pommerenke-92} that $\omega$ is extended to the
eponymous homeomorphism from $\overline{\mathbb D}$ onto $\overline\varOmega$. Let us move the problems \eqref{f0-prob} and \eqref{fn-prob},
$n\ge1$, to the disk $\mathbb D$. Define the functions
\begin{equation*}
F:=f\circ\omega,\qquad H:=h\circ\omega\qquad F_n:=F_n\circ\omega.
\end{equation*}
Then we have
\begin{equation}\label{oper-change}
\mathcal Lf=\frac{1}{|\omega'|^2}\left[\partial\overline\partial F+\tau\partial\left(\frac{\overline{\omega'}}{\omega'}
\partial F\right)\right]=:\mathcal MF.
\end{equation}
Equations \eqref{series}, \eqref{f0-prob} and \eqref{fn-prob} yield
\begin{equation}\label{series-disk}
F=\sum\limits_{n=0}^\infty F_n\tau^n,
\end{equation}
where
\begin{equation}\label{F0-prob}
\partial\overline\partial F_0=0\quad\text{in}\quad\mathbb D,\qquad\qquad F_0|_{\mathbb T}=H
\end{equation}
and
\begin{equation}\label{Fn-prob}
\partial\overline\partial F_n=-\partial\left(\frac{\overline{\omega'}}{\omega'}\partial F_{n-1}\right)\quad\text{in}\quad\varOmega,\qquad\qquad F_n|_{\mathbb T}=0
\end{equation}
for $n\geqslant 1$. Assuming that the functions $h$ and $F_{n-1}$ (for some $n>1$) are sufficiently regular, we can write out the solutions for
\eqref{F0-prob} and \eqref{Fn-prob}:
\begin{equation}\label{induct-disk}
F_0(z)=\frac{1}{2i}\int_{\mathbb T}\partial_\zeta G(\zeta, z)h(\zeta)d\zeta,\qquad
F_n(z)=-\int_{\mathbb D} G(\zeta, z)\partial\left(\frac{\overline{\omega'(\zeta)}}{\omega'(\zeta)}\partial F_{n-1}(\zeta)\right)d\mu,
\end{equation}
$n\geqslant 1$. Here
\begin{equation}\label{Green-func}
G(\zeta, z)=\frac{2}{\pi}\log\left|\frac{\zeta-z}{1-\zeta\overline z}\right|
\end{equation}
is the Green's function for the operator $\partial\overline\partial=\Delta/4$ in the disk $\mathbb D$ and $\mu$ is the planar Lebesgue measure.

Let us define the operator $\mathcal P$ acting on $C(\mathbb T)$ and the operators $\mathcal K$, $\mathcal K_z$, $\mathcal K_{\overline z}$ acting
on $L_p(\mathbb D)$ as follows:
\begin{equation*}\label{ind-oper}
\mathcal P[\varphi(z)]:=\frac{1}{2i}\int_{\mathbb T}\partial_\zeta G(\zeta, z)\varphi(\zeta)d\zeta,\qquad
\mathcal K[\varphi(z)]:=\int_{\mathbb D}\partial_\zeta G(\zeta, z)\varphi(\zeta)d\mu
\end{equation*}
and
\begin{equation*}\label{der-oper}
\mathcal K_z[\varphi(z)]:=\text{v.p.}\int_{\mathbb D}\partial_z\partial_\zeta G(\zeta, z)\varphi(\zeta)d\mu,\qquad
\mathcal K_{\overline z}[\varphi(z)]:=\text{v.p.}\int_{\mathbb D}\partial_{\overline z}\partial_\zeta G(\zeta, z)\varphi(\zeta)d\mu.
\end{equation*}
The symbol $\text{v.p.}\int$ means that the corresponding integrals are understood in the sense of Cauchy principal values. For the sake of
brevity we omit this symbol in what follows everywhere, where it will not cause any misunderstanding. From \eqref{Green-func} we obtain
\begin{equation}\label{ind-oper-2}
\begin{aligned}
\mathcal P[\varphi(z)]=\frac{1}{2\pi}\int_{\mathbb T}\left(\frac{1}{\zeta-z}+\frac{\overline z}{1-\zeta\overline z}\right)\varphi(\zeta)d\zeta,\\
\mathcal K[\varphi(z)]=\frac{1}{\pi}\int_{\mathbb D}\left(\frac{1}{\zeta-z}+\frac{\overline z}{1-\zeta\overline z}\right)\varphi(\zeta)d\mu
\end{aligned}
\end{equation}
and
\begin{equation}\label{der-oper-2}
\mathcal K_z[\varphi(z)]=\frac{1}{\pi}\int_{\mathbb D}\frac{\varphi(\zeta)d\mu}{(\zeta-z)^2},\qquad
\mathcal K_{\overline z}[\varphi(z)]=\frac{1}{\pi}\int_{\mathbb D}\frac{\varphi(\zeta)d\mu}{(1-\zeta\overline z)^2}.
\end{equation}

In terms of the operators defined above, the formulae \eqref{induct-disk} for constructing the functions $F_n$ can be written in the short form
\begin{equation}\label{induct-oper}
F_0=\mathcal P[h],\qquad F_n=\mathcal K[(\overline{\omega'}/\omega')\partial F_{n-1}],\quad n\geqslant 1.
\end{equation}
We also need the partial sums of series \eqref{series} and \eqref{series-disk} respectively:
\begin{equation}\label{part-sum}
s_m=\sum_{n=0}^m f_n\tau^n,\qquad S_m=\sum_{n=0}^m F_n\tau^n.
\end{equation}
Now we are ready to state our main result, which is the following theorem.
\begin{theorem}\label{dirprob-solv}
Let $\varOmega$ be a Jordan domain with the boundary $\varGamma$, let $\omega$ be some conformal mapping from $\mathbb D$ onto
$\varOmega$, and let $h\in C(\varGamma)$. Suppose $\varOmega$ and $h$ are such that $h\circ\omega\in C^\alpha(\mathbb T)$ for some $1/2<\alpha<1$.
Then for every $\tau\in[0, 1)$ the series \eqref{series}, where $f_n=F_n\circ\omega^{-1}$, $n\ge0$, and where $F_n$ are defined by
\eqref{induct-oper}, converges uniformly on $\overline\varOmega$ to the function $f\in C(\overline\varOmega)$ that satisfies the equation
$\mathcal Lf=0$ in $\varOmega$ and coincides with $h$ on $\varGamma$.
\end{theorem}

\section{Auxiliary lemmas and proof of Theorem \ref{dirprob-solv}}

Let us recall, that $L_p(\mathbb D)$ is the usual Lebesgue space of functions considered with respect to the planar Lebesgue measure on
$\mathbb D$, and $W^1_p(\mathbb D)$ is the standard Sobolev space of functions in $\mathbb D$. We start with the following technical lemma.
\begin{lemma}\label{Pois-Wp-lem}
If $\varphi\in C^\alpha(\mathbb T)$, where $0<\alpha<1$, then $\mathcal P\varphi\in W^1_p(\mathbb D)$ with any exponent $0<p<(1-\alpha)^{-1}$.
\end{lemma}

\begin{proof}
Take $\psi=\mathcal P\varphi$. Since $\varphi\in C(\mathbb T)$, then by the properties of the Poisson integral, we have
$\psi\in C(\overline{\mathbb D})$, so obviously $\psi\in L_p(\mathbb D)$. We are going to prove the $L_p$--integrability of the first partial derivatives of $\psi$. We represent the function $\psi$ as the sum $\psi(z)=\psi_1(z)+\overline{\psi_2(z)}$ of holomorphic components
\begin{equation*}
\psi_1(z)=\frac{1}{2\pi i}\int_{\mathbb T}\frac{\varphi(\zeta)d\zeta}{\zeta-z},\qquad
\psi_2(z)=-\frac{1}{2\pi i}\int_{\mathbb T}\frac{\overline{\varphi(\zeta)}d\zeta}{\zeta-z}-
\frac{1}{2\pi}\int_{\mathbb T}\overline{\varphi(\zeta)}|d\zeta|.
\end{equation*}
Since $\varphi\in C^\alpha(\mathbb T)$ with $0<\alpha<1$, then $\psi_1\in C^{\alpha}(\overline{\mathbb D})$, because of the Privalov
theorem \cite{Privalov-39} for Cauchy type integral. Put $T(z, r):=\{\zeta\in\mathbb C\colon |\zeta-z|=r\}\subset\mathbb D$. Using the Cauchy
formula
$$
\psi_1(z)=\frac{1}{2\pi i}\int_{T(z,r)}\frac{\psi_1(\zeta)d\zeta}{\zeta-z}
$$
we obtain the equality
\begin{equation*}
\partial\psi(z)=\psi_1'(z)=\frac{1}{2\pi i}\int_{T(z,r)}\frac{\psi_1(\zeta)d\zeta}{(\zeta-z)^2}=
\frac{1}{2\pi i}\int_{T(z,r)}\frac{\psi_1(\zeta)-\psi_1(z)}{(\zeta-z)^2}d\zeta,
\end{equation*}
and, furthermore, the estimate
$$
|\partial\psi(z)|\leqslant\frac{1}{2\pi i}\int_{T(z,r)}\frac{|\psi_1(\zeta)-\psi_1(z)|}{|\zeta-z|^2}|d\zeta|\leqslant
\frac{1}{2\pi i}\int_{T(z,r)}\frac{c_\alpha|\zeta-z|^\alpha}{|\zeta-z|^2}|d\zeta|=\frac{c_\alpha}{r^{1-\alpha}},
$$
where $c_\alpha=\sup_{\zeta\ne z}|\psi_1(\zeta)-\psi_1(z)|/|\zeta-z|^\alpha$. Taking the limit $r\to(1-|z|)$ we obtain
$|\partial\psi(z)|\leqslant c_\alpha(1-|z|)^{\alpha-1}$, see also \cite[page 74]{Duren-70} or \cite[page 50]{Pommerenke-92}. This means that
$\partial\psi\in L_p(\mathbb D)$ if $(1-\alpha)p<1$. Similarly, the $L_p$--integrability of the derivative  $\overline\partial\psi=\overline{\psi_2'}$ is established under the same condition on $p$. The lemma is proved.
\end{proof}

Consider the operator
\begin{equation}\label{polar-oper}
K\varphi(z):=\frac{1}{\pi}\int_{\mathbb C}\frac{\varphi(\zeta)d\mu}{(\zeta-z)^2},
\end{equation}
which is a bounded one from $L_p(\mathbb C)$ to itself for $p\in(1, \infty)$ by the celebrated theorem of Calderon and Zygmund \cite{CalZig-52}.
We denote by $\|K\|_p$ the norm of the operator $K$ in $L_p(\mathbb C)$ and similarly we denote the norms of operators acting in $L_p(U)$ with
arbitrary domain $U$. It is important for us that $\|K\|_p\to 1$ when $p\to 2$, see \cite[page 89]{Ahlfors-66}, \cite[pages 5--6]{Christ-90}.

\begin{lemma}\label{Lp-oper-pro}
The operators \eqref{ind-oper-2}, \eqref{der-oper-2} have the following properties:

\smallskip
\textup({\rm i}\textup) $\mathcal K\colon L_p(\mathbb D)\to L_p(\mathbb D)$ is bounded for $p>1$;

\smallskip
\textup({\rm ii}\textup) $\mathcal K_z\colon L_p(\mathbb D)\to L_p(\mathbb D)$ is bounded for $p>1$, and $\|\mathcal K_z\|_p=\|K\|_p$;

\smallskip
\textup({\rm iii}\textup) $\mathcal K_{\overline z}\colon L_p(\mathbb D)\to L_p(\mathbb D)$ is bounded for $p\geqslant 2$; moreover,
$\|\mathcal K_{\overline z}\|_p\leqslant\|K\|_p$ and, in particular, $\|\mathcal K_{\overline z}\|_2=\|K\|_2$.
\end{lemma}

\begin{proof}
The statement (i) follows from the fact that the kernel of the integral $\mathcal K$ is the sum of two kernels with a weak singularity, that is they have the growth of the first order.

(ii) Take $\varphi\in L_p(\mathbb D)$, $p>1$. We have $\mathcal K_z\varphi=K\varphi_1$, where the function $\varphi_1$ coinsides with $\varphi$ in the unit disk $\mathbb D$ and equals to zero outside $\mathbb D$, thus $\|\varphi_1\|_{L_p(\mathbb C)}=\|\varphi\|_{L_p(\mathbb D)}$. Therefore,
$\|\mathcal K_z\|_p=\|K\|_p$.

(iii) Making in \eqref{der-oper-2} the change of variables $\zeta\to\xi=1/\overline\zeta$, we obtain
\begin{equation}\label{iii-proof}
\mathcal K_{\overline z}[\varphi(z)]=\frac{1}{\pi}\int_{\mathbb C\setminus\overline{\mathbb D}}\frac{\varphi(1/\overline\xi)}{\xi^2}\cdot
\frac{d\mu}{(\overline\xi-\overline z)^2}=\overline{K[\varphi_2(z)]},
\end{equation}
where $\varphi_2(z)=\overline{\varphi(1/\overline z)}/\overline z^2$ for $z\in\mathbb C\setminus\overline{\mathbb D}$ and $\varphi_2(z)=0$ for $z\in\mathbb D$. Then, by means of reverse change $\xi\to\zeta=1/\overline\xi$ we find
$$
\|\varphi_2\|_{L_p(\mathbb D)}=\left(\int_{\mathbb C\setminus\overline{\mathbb D}}|\varphi_2(\xi)|^p d\mu\right)^{\frac{1}{p}}=
\left(\int_{\mathbb D}|\xi|^{2p-4}\cdot|\varphi(\zeta)|^p d\mu\right)^{\frac{1}{p}}\leqslant\|\varphi\|_{L_p(\mathbb D)}
$$
for $p\geqslant 2$ with the equality for $p=2$. Hence, it follows from \eqref{iii-proof} that
$\|\mathcal K_{\overline z}\|_{L_p(\mathbb D)}\leqslant\|K\|_p\cdot\|\varphi\|_{L_p(\mathbb D)}$, i.e.
$\|\mathcal K_{\overline z}\|_p\leqslant\|K\|_p$ for $p\geqslant 2$ with equality for $p=2$. The proposition is proved.
\end{proof}

\begin{proof}[Proof of Theorem \ref{dirprob-solv}]
First we are going to establish the convergence of the series \eqref{series-disk} and its partial derivatives of the first order
in the norm of $L_p(\mathbb D)$. Lemma \ref{Pois-Wp-lem} yields that $F_0=\mathcal Ph$ belongs to the Sobolev space $W^1_p(\mathbb D)$ for $p<(2(1-\alpha))^{-1}$. Suppose that $F_{n-1}\in L_p(\mathbb D)$, $p>2$, for some number $n$. Then
$$
\partial F_n=\mathcal K_z[(\overline{\omega'}/\omega')\partial F_{n-1}],\qquad
\overline\partial F_n=(\mathcal I+\mathcal K_{\overline z})[(\overline{\omega'}/\omega')\partial F_{n-1}]
$$
in the sense of distributions (see \cite[page 90]{Ahlfors-66}). Using these equalities and applying Lemma~\ref{Lp-oper-pro}, we derive from \eqref{induct-oper} that the following estimates hold
$$
\|\partial F_n\|_{L_p(\mathbb D)}\leqslant\|K\|_p\cdot\|\partial F_{n-1}\|_{L_p(\mathbb D)},\qquad
\|\overline\partial F_n\|_{L_p(\mathbb D)}\leqslant(1+\|K\|_p)\|\partial F_{n-1}\|_{L_p(\mathbb D)}.
$$
This implies
\begin{equation}\label{der-Fn-est}
\|\partial F_n\|_{L_p(\mathbb D)}\leqslant\|K\|_p^n\cdot\|\partial F_0\|_{L_p(\mathbb D)},\quad
\|\overline\partial F_n\|_{L_p(\mathbb D)}\leqslant(1+\|K\|_p)\|K\|_p^{n-1}\cdot\|\partial F_0\|_{L_p(\mathbb D)}.
\end{equation}
Then \eqref{induct-oper} yields
\begin{equation}\label{Fn-est}
\|F_n\|_{L_p(\mathbb D)}\leqslant\|P\|_p\cdot\|\partial F_{n-1}\|_{L_p(\mathbb D)}\leqslant\|P\|_p\cdot\|K\|_p^{n-1}\cdot\|\partial F_0\|_{L_p(\mathbb D)}.
\end{equation}

Estimate \eqref{Fn-est} gives the convergence of the series \eqref{series-disk} in $L_p(\mathbb D)$ to its sum $F\in L_p(\mathbb D)$, when
$\|K\|_p\tau<1$. Moreover,
\begin{multline}\label{F-est}
\|F\|_{L_p(\mathbb D)}=\left\|\sum_{n=0}^\infty F_n\tau^n\,\right\|_{L_p(\mathbb D)}\leqslant\|F_0\|_{L_p(\mathbb D)}+
\sum_{n=1}^\infty\|F_n\|_{L_p(\mathbb D)}\tau^n\leqslant\\
\leqslant\|F_0\|_{L_p(\mathbb D)}+\sum_{n=1}^\infty\|P\|_p\cdot\|K\|_p^{n-1}\cdot\|\partial F_0\|_{L_p(\mathbb D)}\cdot\tau^n=\\
=\|F_0\|_{L_p(\mathbb D)}+\frac{\tau\|P\|_p}{1-\tau\|K\|_p}\|\partial F_0\|_{L_p(\mathbb D)}.
\end{multline}
Estimate \eqref{der-Fn-est} gives the convergence of the first partial derivatives of the series \eqref{series-disk} to the corresponding derivatives of $F$ in the same space $L_p(\mathbb D)$:
\begin{equation}\label{der-F-est}
\|\partial F\|_{L_p(\mathbb D)}\leqslant\frac{\|\partial F_0\|_{L_p(\mathbb D)}}{1-\|K\|_p\tau},\qquad
\|\overline\partial F\|_{L_p(\mathbb D)}\leqslant\|\overline\partial F_0\|_{L_p(\mathbb D)}+
\frac{\tau(1+\|K\|_p)\,\|\partial F_0\|_{L_p(\mathbb D)}}{1-\tau\|K\|_p}.
\end{equation}
The obtained estimates \eqref{F-est} and \eqref{der-F-est} mean the convergence of the series \eqref{series-disk} in the norm of the Sobolev space
$W_p^1(\mathbb D)$:
\begin{equation}\label{ser-conv}
\lim\limits_{m\to\infty}\|F-S_m\|_{W_p^1(\mathbb D)}=0.
\end{equation}
By virtue of Sobolev embedding theorem \cite{Sobolev-88}, we have $W_p^1(\mathbb D)\subset C(\overline{\mathbb D})$ when $1-2/p>0$, or,
equivalently, $p>2$, and the embedding is compact. As $F\in W_p^1(\mathbb D)$, so $F\in C(\overline{\mathbb D})$ and thus
$f=F\circ\omega^{-1}\in C(\overline\varOmega)$. In view of the compactness of the embedding, the series \eqref{series-disk}, and therefore
\eqref{series}, converge uniformly in $\overline{\mathbb D}$ and $\overline\varOmega$ respectively.

\vspace*{3mm}

Now let us prove that the function $f=F\circ\omega^{-1}$ satisfies the equation $\mathcal Lf=0$ in $\varOmega$ by establishing it
firstly in the sense of distributions.

Let $\phi$ be arbitrary test function of the class $C^2_0(\mathbb D)$ of twice continuously differentiable functions in $\mathbb C$ with compact
support in $\mathbb D$ and let
$$
\langle g|\phi\rangle:=\int_{\mathbb C} g(z)\phi(z)d\mu
$$
be the action of the distribution $g$ on $\phi$. We have
\begin{multline*}
\langle F_n|\partial\overline\partial\phi\rangle=\int_{\mathbb D}\partial\overline\partial\phi(z)d\mu(z)
\int_{\mathbb D}\partial_\zeta G(\zeta, z)\frac{\overline{\omega'(\zeta)}}{\omega'(\zeta)}\partial F_{n-1}(\zeta)d\mu(\zeta)=\\
=\int_{\mathbb D}\frac{\overline{\omega'(\zeta)}}{\omega'(\zeta)}\partial F_{n-1}(\zeta)d\mu(\zeta)\partial_\zeta
\int_{\mathbb D} G(\zeta, z)\partial\overline\partial\phi(z)d\mu(z)=\\
=\int_{\mathbb D}\frac{\overline{\omega'(\zeta)}}{\omega'(\zeta)}\partial F_{n-1}(\zeta)\partial\phi(\zeta)d\mu(\zeta)=
\langle(\overline{\omega'}/\omega')\partial F_{n-1}|\partial\phi\rangle,
\end{multline*}
which means that the equality $\partial\overline\partial F_n=-\partial[(\overline{\omega'}/\omega')\partial F_{n-1}]$ holds in $\mathbb D$ in the
sense of distributions. This together with \eqref{oper-change} implies the following chain of equalities that are understood in the sense of
distributions:
$$
\begin{aligned}
|\omega'|^2\mathcal Ls_m&=\sum_{n=0}^m\left(\partial\overline\partial F_n+\tau\partial\left(\frac{\overline{\omega'}}{\omega'}
\partial F_n\right)\right)\tau^n=\\
&=\partial\overline\partial F_0+\sum_{n=1}^m\left(\partial\overline\partial F_n+\partial\left(\frac{\overline{\omega'}}{\omega'}
\partial F_{n-1}\right)\right)\tau^n+\partial\left(\frac{\overline{\omega'}}{\omega'}\partial F_m\right)\tau^{m+1}=\\
&=\partial\left(\frac{\overline{\omega'}}{\omega'}\partial F_m\right)\tau^{m+1}.
\end{aligned}
$$
Thus for each function $\varphi\in C_0^2(\varOmega)$, one can write
\begin{equation*}
\langle\mathcal L s_m\,|\,\varphi\rangle=\langle\partial\left[(\overline{\omega'}/\omega')\partial F_m\right]\,|\,\phi\rangle\cdot\tau^{m+1}=
-\langle\partial F_m\,|\,(\overline{\omega'}/\omega')\partial\phi\rangle\cdot\tau^{m+1},
\end{equation*}
where $\phi:=\varphi\circ\omega\in C_0^2(\mathbb D)$. Then
\begin{multline*}
\langle\mathcal Lf\,|\,\varphi\rangle:=\langle f|\mathcal L\varphi\rangle=\langle f-s_m|\mathcal L\varphi\rangle+\langle s_m|
\mathcal L\varphi\rangle=\\
=\langle F-S_m|\mathcal M\phi\rangle-\langle\partial F_m|(\overline{\omega'}/\omega')\partial\phi\rangle\cdot\tau^{m+1}.
\end{multline*}
Take $1/p+1/q=1$. Applying the H\"older inequality ant taking into account \eqref{der-Fn-est} and \eqref{ser-conv} we obtain
\begin{multline*}
\left|\langle\mathcal Lf\,|\,\varphi\rangle\right|\leqslant \|F-S_m\|_{L_p(\mathbb D)}\cdot\|\mathcal M\phi\|_{L_q(\mathbb D)}+
\|\partial F_m\|_{L_p(\mathbb D)}\cdot\|\partial\phi\|_{L_q(\mathbb D)}\cdot\tau^{m+1}\leqslant\\
\leqslant\|F-S_m\|_{L_p(\mathbb D)}\cdot\|\mathcal M\phi\|_{L_q(\mathbb D)}+\|\partial F_0\|_{L_p(\mathbb D)}\cdot
\|\partial\phi\|_{L_q(\mathbb D)}\cdot\|K\|_p^m\tau^{m+1}\to 0
\end{multline*}
when $m\to\infty$ and $\|K\|_p\tau<1$. Thus $\langle\mathcal Lf\,|\,\varphi\rangle=0$, i.e the function $f$ satisfies the equation
$\mathcal Lf=0$ in $\varOmega$ in the sense of distributions. Since this equation is of elliptic type, by virtue of Wheil's lemma it also holds
in the classic sense.

\vspace*{3mm}

It remains to show that $f|_\Gamma=h$. Note that the estimates \eqref{der-Fn-est} and \eqref{Fn-est} yield $F_n\in W_p^1(\mathbb D)$.
If $p>2$, by the Sobolev embedding theorem, this implies that $F_n\in C(\overline{\mathbb D})$. Since $F_0$ is the harmonic extension of the
boundary function $H\in C^\alpha(\mathbb T)$, then $F_0|_{\mathbb T}=H$. The rest of the functions $F_n$ vanish on $\mathbb T$: it follows from
equations \eqref{induct-disk} and from the continuity of $F_n$ in $\overline{\mathbb D}$. Thus $S_m|_\mathbb T=H$ for all $m$. The compactness of
the embedding $W_p^1(\mathbb D)\subset C(\overline D)$ and the convergence \eqref{ser-conv} lead to uniform convergence
$\|F-S_m\|_{C(\overline{\mathbb D})}\to 0$, so that $F|_\mathbb T=S_m|_\mathbb T=H$. Consequently, $f|_\Gamma=h$.

The given arguments are valid when $2<p<(1-\alpha)^{-1}$, which is true since it is assumed that $1/2<\alpha<1$. As $\|K\|_p\to 1$ when
$p\to 2$, then for each fixed value of $\tau<1$ one can choose a value of $p>2$ sufficiently close to $2$ such that $\|K\|_p\tau<1$. This
finishes the proof of the theorem.
\end{proof}

\section*{Acknowledgement(s)}

I am grateful to professors K.Yu.~Fedorovskiy and A.P.~Soldatov for helpful discussions.

\end{document}